\documentclass[12pt]{amsart}
\usepackage{amsmath}
\usepackage{latexsym}
\usepackage{amssymb}
\usepackage{graphicx}
\DeclareMathOperator{\co}{co}

\DeclareMathOperator{\supp}{supp}

\newtheorem{theorem}{Theorem}

\newtheorem{lemma}{Lemma}

\newtheorem{example}{Example}

\newtheorem{remark}{Remark}

\begin{document}
%
\def\R {{\mathbb{R}}}
\def\N {{\mathbb{N}}}
\def\C {{\mathbb{C}}}
\def\Z {{\mathbb{Z}}}
\def\phi{\varphi}
\def\epsilon{\varepsilon}
\def\ma{{\mathcal A}}
%
\def\tb#1{\|\kern -1.2pt | #1 \|\kern -1.2pt |} 
\def\Qed{\qed\par\medskip\noindent}
%

\title[On the extension of semiconcave functions]{On the extension problem for semiconcave functions with fractional modulus}  
\author{Paolo Albano} 
\address{Dipartimento di Matematica, Universit\`a di Bologna, Piazza di Porta San Donato 5, 40127 Bologna, Italy} 
\email{paolo.albano@unibo.it}
\author{Vincenzo Basco}
\address{Thales Alenia Space, via Saccomuro 24, Roma, Italy.}
\email{vincenzo.basco@thalesaleniaspace.com}
\email{vincenzobasco@gmail.com}
\author{Piermarco Cannarsa} 
\address{Dipartimento di Matematica, Universit\`a di Roma "Tor Vergata",  Roma, Italy} 
\email{cannarsa@mat.uniroma2.it }

\date{\today}

\begin{abstract}
Consider a locally Lipschitz function $u$ on the closure of  a possibly unbounded open subset $\Omega$ of $\R^n$ with nonempty boundary. Suppose $u$ is  (locally) semiconcave on $\overline \Omega$ with a fractional semiconcavity modulus.  Is it possible to extend $u$  in a neighborhood of any boundary point retaining the same semiconcavity modulus?  We show that  this is indeed the case and we give two applications of  this extension property. First, we derive an approximation result  for semiconcave functions on  closed domains.  Then, we use the above extension property to study the propagation of singularities of semiconcave functions at  boundary points.
\end{abstract}


\maketitle

\section{Introduction and statement of the results}
\setcounter{equation}{0}
\setcounter{theorem}{0}
\setcounter{proposition}{0}  
\setcounter{lemma}{0}
\setcounter{corollary}{0} 
\setcounter{definition}{0}
Semiconcave functions are an important class of nonsmooth functions that relaxes the classical notion of concavity allowing for localization, but preserves the main properties of concave functions. Such a class has many applications in analysis and geometry, see \cite{CS} and \cite{V}.

When semiconcave functions are assumed to have linear modulus, their structure is particularly appealing as they can be locally represented as the sum of a concave function plus a smooth remainder. However, such a representation fails if the modulus is fractional, even though many other  properties remain true. On the other hand, fractionally semiconcave functions are interesting in their own right. For instance,  value functions of constrained variational problems have been proved to be semiconcave with fractional modulus (see \cite{CCC2}), while linear semiconcavity cannot be expected (see \cite{CM}).

This paper is concerned with the classical extension problem, that we set and solve in the class of fractionally semiconcave functions.  

We observe that, in the case of a semiconcave functions with linear modulus, a local extension property, near a boundary point of a convex set, was established in \cite[Proposition 3.1]{A2}.

We give two extension results: a local one, showing that in this case the set of reachable gradients is preserved, and a global one. Then, we deduce an approximation theorem for a semiconcave function on a closed domain and we study the propagation of singularities from a boundary point.

We now proceed to describe the problem in detail and outline our main results.

 \subsection{Semiconcave functions with fractional modulus}
Let $A\subset \R^n$ ($n\geq 1$),  a function $u:A\longrightarrow \R$ is semiconcave with a fractional modulus of semiconcavity if it is locally Lipschitz continuous\footnote{We observe that if $A$ is an open set, then $u$ is locally Lipschitz continuous on $A$ (see e.g. \cite{CS}). We are requiring the local Lipschitz regularity of $u$ since we are not assuming that $A$ be an open set.} and there exist $\alpha \in ]0,1]$ and $C=C_A\in \mathbb R$ such that 
\begin{equation}\label{eq:sca}
\lambda u(x)+(1-\lambda) u(y)-u(\lambda x + (1-\lambda )y)\le C\lambda (1-\lambda) |x-y|^{1+\alpha},
\end{equation}
for any $x,y\in A$ such that the line segment $[x,y]$ is contained in $A$ and for every $\lambda\in [0,1]$.
Furthemore, we call any constant $C\in \mathbb R$, for which \eqref{eq:sca} holds true, a semiconcavity constant for $u$ in $A$. 
Finally, we define 
\begin{multline*}
SC_{loc}^\alpha (A)=\{ u:A\longrightarrow \R\, :\,u\text{ is locally Lipschitz on } A\text{ and}
\\
u\text{ satisfies \eqref{eq:sca}  on every compact subset of }A\}  
\end{multline*}
(in the above definition  the constant $C$ may depend on the compact set under exam).

 \subsection{The extension problem}

For $x_0\in \R^n$ and $\delta >0$, we denote by $B_\delta (x_0)$ the open ball with center at $x_0$ and radius $\delta$. Let $\Omega\subset \R^n$ be an open set with nonempty boundary.

We study the following problem: given $u\in SC_{loc}^\alpha (\overline{\Omega})$ and a boundary point $x_0\in \partial\Omega$, is there a function $E(u)$ defined on $B_\delta (x_0)$ (for a suitable $\delta >0$),  such that  
$$
E(u)\in SC^\alpha (B_\delta (x_0))\text{ and }E(u)(x)=u(x), \text{ for every }x\in B_\delta (x_0)\cap \overline{\Omega}? 
$$

For some of the applications we have in mind we will need  $E(u)$ to satisfy the additional condition we are going to describe. For $x\in \overline{\Omega}$, we define  the set of reachable (or, achievable) gradients of $u$ at $x$ by
\begin{multline}\label{eq:dstar}
D^*u(x)=\{ p\in\R^n~:~p=\lim_{h\to \infty} Du(x_h)\,,\;\Omega\ni x_h\to x\,, 
\\
\text{$u$ differentiable at $x_h$}\}. 
\end{multline}
Notice  that $D^*u(x)\not=\emptyset$
for all $x\in \overline{\Omega}$ owing to Rademacher's Theorem.

Furthermore, for every $x\in \overline{\Omega}$ and for every $p\in D^*u(x)$, 
\begin{equation}\label{p}
u(y)\le u(x)+\langle p,y-x\rangle +C |y-x|^{1+\alpha }, 
\end{equation}
for every $ y\in \overline{\Omega}$ such that $[x ,y]\subset \overline{\Omega}$. 
Then, we would like to find a local extension of $u$ with the additional property 
\begin{equation}\label{eq:dstarp}
D^*E(u)(x)=D^*u(x),\quad \forall x\in B_\delta (x_0)\cap \partial\Omega.
\end{equation} 
\begin{remark}
We point out that  property \eqref{eq:dstarp} is useful in the analysis of the singularities (i.e. the points of nondifferentiability) of $u$. 
\end{remark}
Let us state our local extension result:
\begin{theorem}\label{t1}
Let $\Omega\subset \R^n$ be an open set with nonempty boundary, and let $u\in SC_{loc}^\alpha (\overline{\Omega})$. Then, for every $x\in \partial\Omega$ and for every $\delta>0$ there exist a function $E(u)\in SC^\alpha (B_\delta (x))$ such that 
\begin{enumerate} 
\item $E(u)(y)=u(y)$ for every $y\in B_\delta (x)\cap \overline{\Omega}$;
\item $D^*E(u)(y)=D^*u(y)$ for every $y\in B_\delta (x)\cap \partial\Omega$.
\end{enumerate}
Furthermore, denoting by $C$ a semiconcavity constant for $u$ in $B_\delta (x)\cap \overline{\Omega}$, we have that there is $C_E$, a semiconcavity constant for $E(u)$ in $B_\delta (x)$, with 
\begin{equation}\label{eq:estce}
C_E\le (C+1)(1+\alpha )(1+2^{2-\alpha} ).
\end{equation}
\end{theorem}
\begin{remark}

\

(i) A  first version of this extension result was established in \cite{A2} in the special case of $\alpha =1$ with $\Omega$ convex. While the result in \cite{A2} was motivated by the study of the 
boundary behavior of the singularities of solutions to first order Hamilton-Jacobi equations with Dirichlet boundary conditions, Theorem \ref{t1} is 
intended for different boundary data such as those which fit optimal control problems with state constraints.

(ii)  The extension $E(u)$ given in the proof of Theorem \ref{t1} is not unique (see Remark \ref{r:ext} below).

\end{remark} 
Given a set $A$ we denote by $\co A$ the convex hull of $A$.

A global version of Theorem \ref{t1} is obtained as follows: 
\begin{theorem}\label{t1*}
Let $\Omega\subset \R^n$ be an open set with nonempty boundary, and let $u\in SC_{loc}^\alpha (\overline{\Omega})$. Then, there exist an open subset $\Omega'\supset \Omega$ and a function $E(u)\in SC_{loc}^\alpha (\Omega')$  such that

\begin{enumerate} 
\item $E(u)(x)=u(x)$ for every $x\in\overline{\Omega}$;
\item $\co D^*E(u)(x)=\co D^*u(x)$ for every $x\in \partial\Omega$.
\end{enumerate}

\end{theorem}

\begin{remark}
We observe that, in   Theorem~\ref{t1*}, $\Omega$ may be unbounded. 
If we assume that $\Omega$ is a bounded set then one can obtain some stronger results: 
\begin{itemize}
\item[$(i)$] taking $x_0\in \partial \Omega$ and $\delta >0$ so that $\overline{\Omega}\subset B_\delta (x_0)$, and using Theorem~\ref{t1} we find an extension of $u$, $E(u)$, such that 
$$
D^*E(u)(x)=D^*u(x),\qquad \forall x\in \partial\Omega , 
$$
instead of the weaker conclusion (2) in Theorem~\ref{t1*}.

\item[$(ii)$] $u$ can be extended as a semiconcave function on the whole space. (This property can be justified by a cut-off argument.)
\end{itemize}
\end{remark}
\subsection{Approximation}
A first application of Extension Theorem \ref{t1} is the following local approximation result 
\begin{theorem}\label{app}
Let $u\in SC_{loc}^\alpha (\overline{\Omega})$ and let $x_0\in\partial\Omega$. Then there exist $\delta>0$ and a sequence of functions $u_h\in C_0^\infty (B_\delta (x_0))$, $h\in \N$, such that 
\begin{enumerate}
\item $u_h$ uniformly converges to $u$ on $B_{\delta/2}(x_0)\cap \bar{\Omega}$ as $h\to \infty$;
\item denoting by $C$  a semiconcavity constant for an extension of $u$ on $B_\delta (x_0)$, we have that
\end{enumerate} 
\begin{multline*}
\lambda u_h(x_1)+(1-\lambda )u_h(x_2) -u_h(\lambda x_1 +(1-\lambda )x_2 )
\\
\leq C \lambda (1-\lambda ) |x_1-x_2|^{1+\alpha },
\end{multline*}
for every $\lambda \in [0,1]$ and $x_1,x_2 \in B_{\delta/2}(x_0)$.
\end{theorem}
We observe that a global version of the above result  can be given as follows. 
\begin{theorem}\label{app*}
Let $u\in SC_{loc}^\alpha (\overline{\Omega})$. Then there exist an open set $\Omega'\supset \overline{\Omega}$ and a sequence of functions $u_h\in C^\infty (\Omega')$, $h\in \N$, such that 
\begin{enumerate}
\item $u_h$ uniformly converges to $u$  locally uniformly on $\overline{\Omega}$, as $h\to \infty$;
\item for every compact set $K\subset \Omega'$ denoting by $C$  a semiconcavity constant, on $K$,  for the extension of $u$ given by Theorem \ref{t1*}, we have that 
\end{enumerate} 
\begin{multline*}
\lambda u_h(x_1)+(1-\lambda )u_h(x_2) -u_h(\lambda x_1 +(1-\lambda )x_2 )
\\
\leq C \lambda (1-\lambda ) |x_1-x_2|^{1+\alpha },
\end{multline*}
for every $\lambda \in [0,1]$ and $x_1,x_2 \in K$ such that $[x_1,x_2]\subset K$.
\end{theorem}

\begin{remark}
It is worth noting that, in the case of $\alpha=1$, Theorem~\ref{app*} ensures the existence of an approximating sequence satisfying, 
on every compact set $K\subset \Omega'$,
\begin{equation*}
D^2u_h\leq C\cdot I\qquad \forall h\in\N,
\end{equation*}
where $C$ is a semiconcavity constant for $E(u)$ on $K$.
\end{remark}
 \subsection{Singularities}
 Let us introduce the singular set of $u$:   
$$
\Sigma (u)=\{ x\in \overline{\Omega}\ :\ D^*u(x) \text{ has at least two elements}\}. 
$$
We observe that, for $x\in  \Omega$, $x\in \Sigma (u)$ if and only if $u$ is not differentiable at $x$. 
 Given a set $A$ the symbol $\partial A$ stands for the topological boundary of $A$. 
  
We have the following 
 \begin{theorem}\label{sing} 
Let  $\Omega$ be an open set with nonempty boundary, let  $u\in SC^\alpha_{loc} (\overline{\Omega})$ and let $x_0\in \partial \Omega$ such that 
\begin{equation}\label{eq:h}
\partial \co D^*u(x_0)\setminus D^*u(x_0)\not=\emptyset .
\end{equation}
Let $p_0\in \co D^*u(x_0)\setminus D^*u(x_0)$ and let $-\theta$ be a vector in the normal cone to $\co D^*u(x_0)$ at $p_0$. 
Let $E(u)\in SC^\alpha (B_\delta (x_0))$ be an extension of $u$ satisfying property $(2)$ of Theorem~\ref{t1}.   
Then, there is a map 
$$
[0,\sigma ]\ni s\mapsto x(s)\in B_\delta (x_0)
$$
(depending on $E(u)$)  such that 
\begin{enumerate}
\item $x(0)=x_0$ and $\lim_{s\to 0^+} x(s)=x_0$;  
\item   $x(s)\not=x_0$, for every $s\in [0,\sigma ]$;
\item $x(s)\in \Sigma (E(u))$, for every $s\in [0,\sigma ]$;
\item $x(s)=x_0+s\theta  +o(s)$ with $o(s)/s\to 0$ as $s\to 0^+$, 
\end{enumerate}
for a suitable $\sigma>0$ depending on the "initial" point $x_0$.  
\end{theorem}
Since $u$ is semiconcave, the set $\co D^*u(x)$ above coincides with the superdifferential of $u$ at $x$, $D^+u(x)$, for all $x\in \Omega$. See \cite{CS} for more details.

%

 %

\section{Proof of Theorems~\ref{t1} and \ref{t1*}}
\bigskip

\setcounter{equation}{0}
\setcounter{theorem}{0}
\setcounter{proposition}{0}  
\setcounter{lemma}{0}
\setcounter{corollary}{0} 
\setcounter{definition}{0}
\subsection{Proof of Theorem \ref{t1}}
 
 Let $x_0\in \partial\Omega$, $\delta >0$, and set 
 $$
A:=\Omega \cap B_\delta (x_0).  
 $$
Let $C>0$ be a semiconcavity constant for $u$ in $\bar{A}$ and  define 
\begin{equation}\label{eu}
E(u)(x)=
\inf_{y\in \bar{A},\, p\in D^*u(y)} \left [ u(y)+\langle p, x-y\rangle +(C+1)|x-y|^{1+\alpha }\right ],
\end{equation} 
for $x\in B_\delta  (x_0)$.  
We claim that $E(u)$ satisfies all the properties stated in Theorem \ref{t1}. 

\vspace{0,3cm} 
{\sc  (i) $E(u)$ is a semiconcave function in $SC^\alpha (B_\delta  (x_0))$.}  

\smallskip
Indeed, by the definition, we have that 
\begin{enumerate} 
\item $C^{1,\alpha}\subset SC^\alpha $ (this is a direct consequence of the Taylor formula, the semiconcavity constant $C$ can be taken greater of equal to the H\"older seminorm of the gradient);
\item  if $I$ is a  set of indices and $\{ f_i\}_{i\in I}$ is a family of semiconcave functions satisfying \eqref{eq:sca} with the same $C$, then $\inf_{i\in I} f_i$ is a semiconcave function with the same constant $C$ (this follows by \eqref{eq:sca} and the definition of infimum);  
\item the sum of semiconcave functions is a semiconcave function (also this fact is a direct consequence of the definition \eqref{eq:sca}). 
\end{enumerate} 
Now, let us define 
$$
K=\{ (y,p)\ : \  y\in \bar{A}\,\text{ and } p\in D^*u(y) \}.
$$
We observe that $K$ is a compact nonempty set. Indeed $K$ is a bounded set ($D^*u$ is bounded, by the local Lipschitz continuity of $u$) and it is a closed set (if $(y_h,p_h)\in K$ converges to $(y,p)$ then $y\in \bar{A}$ and, by a diagonal argument based on the Rademacher Theorem, we deduce that there exists a sequence $y_j\in   \bar{A}\setminus \Sigma (u)$, such that $(y_j,Du(y_j))\to (y,p)$, i.e. $(y,p)\in K$).

Then, for every $(y,p)\in K$ the function $B_\delta (x_0)\ni x\mapsto u(y)+\langle p, x-y\rangle$ is concave while, for every $y\in \bar{A}$,  the function $B_\delta (x_0)\ni x\mapsto C|x-y|^{1+\alpha}$ is of class $C^{1,\alpha}$.  Let us estimate the H\"older seminorm of the derivative of the last function. For every $y\in  \bar{A}$, set 
$$
v_y(x)=C|x-y|^{1+\alpha}\quad (x\in B_\delta (x_0)).
$$
Then, 
$$
Dv_y(x)=
\begin{cases}
C(1+\alpha ) |x-y|^{\alpha -1} (x-y),\qquad &\text{ if }x\in B_\delta  (x_0)\setminus \{ y\},
\\
0,\qquad &\text{ if }x=y,
\end{cases}
$$
and, for every $x,z\in  B_\delta (x_0)$, we have that 
\begin{multline}\label{eq:dx-dz}
Dv_y(x)-Dv_y(z)\\
=C(1+\alpha ) \left (  |x-y|^{\alpha -1} (x-y)-|z-y|^{\alpha -1} (z-y)\right ).
\end{multline}
We may assume that 
\begin{equation}\label{in}
|x-y|\le |z-y|.
\end{equation}
If $x=y$, by \eqref{eq:dx-dz}, we deduce that 
\begin{equation}\label{eq:parziale}
|Dv_y(x)-Dv_y(z)|\le C(1+\alpha ) |x-z|^{\alpha}.
\end{equation}
Furthermore, for any $x\not= y$, we have 
\begin{align*}
&|Dv_y(x)-Dv_y(z)|\\
&=C(1+\alpha )  \left | (|x-y|^\alpha -|z-y|^\alpha) \frac {x-y}{|x-y|}+|z-y|^\alpha \left (   \frac {x-y}{|x-y|}-\frac {z-y}{|z-y|}\right )\right |
 \\ 
&\le C(1+\alpha )\left ( |x-z|^\alpha + \frac{|z-y|^\alpha }{  |x-y|\, |z-y|  }   [ |z-y| (x-y)-|x-y|(z-y) ]\right ).
\end{align*}
Now, using the elementary inequality 
\begin{align*}
&| |z-y| (x-y)-|x-y|(z-y) |\\
&=| (|z-y|-|x-y|) (x-y)+|x-y|(x-z) |
\\
&\le2|x-y|\, |x-z|,
\end{align*}
we find that 
$$
|Dv_y(x)-Dv_y(z)|\le C(1+\alpha )|x-z|^\alpha \left \{1+  2\left (\frac{|x-z|}{|z-y|}\right )^{1-\alpha}
\right \}.
$$
By \eqref{in}, we have that 
$$
\frac{|x-z|}{|z-y|}\le \frac{|x-y|+|y-z|}{|z-y|}\le 2
$$
and, recalling \eqref{eq:parziale}, we conclude that 
$$
|Dv_y(x)-Dv_y(z)|\le C(1+\alpha )(1+2^{2-\alpha})|x-z|^\alpha ,\quad \forall x,z\in B_\delta (x_0) . 
$$
Hence, by (1), (3) and (2) above, we deduce that $E(u)\in SC^\alpha (B_\delta (x_0))$ and that $(C+1)(1+\alpha )(1+2^{2-\alpha})$ is a semiconcavity constant for $E(u)$ on $B_\delta  (x_0)$.

\vspace{0,3cm} 
{\sc  (ii) $E(u)(x)=u(x)$, for every  $x\in  \bar{A}$.} 

\smallskip
Indeed, let $x\in \bar{A}$. Then, by taking $y=x$ in the definition of $E(u)$, we find that 
$E(u)(x)\le u(x)$. On the other hand,  by \eqref{p}, 
$$
u(x)\le u(y)+\langle p, x-y\rangle +(C+1)|x-y|^{1+\alpha }, 
$$
for every $y\in \bar{A}$, i.e. $u(x)\le E(u)(x)$ and the conclusion follows.

\vspace{0,3cm} 
{\sc  (iii) $D^*E(u)(x)=D^*u(x)$, for every  $x\in B_\delta  (x_0)\cap \partial \Omega$.} 

\smallskip
It suffices to show that $D^*E(u)(x)\subseteq D^*u(x)$ for every  $x\in B_\delta (x_0)\cap \partial \Omega$. 
For this purpose, fix $x\in B_\delta (x_0)\cap \partial \Omega$ and let $p\in D^*E(u)(x)$. Without loss of generality, we
suppose  there is a sequence $x_h\in B_\delta (x_0)\setminus \bar{\Omega}$, $h\in \mathbb N$, such that $x_h\to x$, $E(u)$ is differentiable at $x_h$, and $DE(u)(x_h)\to p$ (indeed, for otherwise the conclusion would be  trivial). 
We have to show that $p\in D^*u(x)$. 

Since $K$ is a compact set, there exists $(y_h,p_h)\in K$ such that 
\begin{enumerate}
\item $(y_h,p_h)\to (y_0,p_0)\in K$, for a suitable $(y_0,p_0)$;
\item $E(u)(x_h)=u(y_h)+\langle p_h, x_h-y_h\rangle +(C+1)|x_h-y_h|^{1+\alpha }$;
\item $DE(u)(x_h)=p_h+(C+1)(1+\alpha )|x_h-y_h|^{\alpha -1}(x_h-y_h)$.
\end{enumerate}
Then, by (1) and (3) above, we find that 
\begin{equation}\label{p1}
p=p_0+(C+1)(1+\alpha )|x -y_0|^{\alpha -1}(x-y_0)
\end{equation}
and, by (2), 
\begin{equation}\label{ad}
u(x)=E(u)(x)=u(y_0)+\langle p_0, x-y_0\rangle +(C+1)|x-y_0|^{1+\alpha }.
\end{equation}
Then, by \eqref{ad} and \eqref{p}, we deduce that 
$$
u(x)\geq u(x)+|x-y_0|^{1+\alpha } \implies x=y_0.
$$
Hence, by \eqref{p1}, we find that $p=p_0\in D^*u(x)$. 
This completes our proof. 

\begin{remark}\label{r:ext} 
We observe that the extension, $E(u)$, given by Formula \eqref{eu}, is not unique. Indeed, the coefficient  $C+1$ in the definition of $E(u)$ is not uniquely determined ($C+1$ can be replaced by any  number greater than a given semiconcavity constant for $u$). 
\end{remark}

\subsection{Proof of Theorem \ref{t1*}}

Not surprisingly, the idea of the proof consists of using a partition of  unity in order to glue together the local extensions given by Theorem \ref{t1}. 
However,  this procedure  provides just a global semiconcave extension of $u$, but yields no information on $D^*E(u)$ at boundary points.
For this purpose, we need an extra argument which represents the main point of the proof.

 We observe that, by Theorem \ref{t1}, for every $x\in\partial\Omega$ there exist $B_{\delta_x} (x)$ and an extension $E_x(u)$ on $B_{\delta_x} (x)$ such that $D^*u(y)=D^*E_x(u)(y)$, for every $y\in \partial\Omega \cap  B_{\delta_x} (x)$. 
 Define 
 $$
 \Omega' :=\cup_{x\in \partial \Omega}B_{\delta_x}(x)\cup \Omega .
 $$
Then $\{B_{\delta_x}(x)\}_{x\in\partial \Omega}\cup \{ \Omega\}$ is a covering of $\Omega'$. 
Let $\{\chi_j\}_{j\in \N}$ be a partition of unity subordinate to such a covering, that is, a countable family of smooth functions  satisfying the following: 
\begin{itemize}
\item[(A)] $0\le \chi_j\le 1$ for all $j$ and all $y\in \Omega'$;
\item[(B)] every $y\in \Omega'$ has a neighborhood  on which all but finitely many functions
$\chi_j$  are identically zero;
\item[(C)]  each function $\chi_j$ is identically zero except on some closed set contained in one of the open sets of the cover; 
\item[(D)] $\sum_j \chi_j (y)=1$ for every $y\in \Omega'$. 
\end{itemize} 
(The existence of $\{\chi_j\}_{j\in \N}$  is well-known, see, e.g.,  page 52 of \cite{GP}.) 

Let 
$$
J=\{ j\in \N\ | \  \exists x\in\partial \Omega :\, \supp \chi_j\subset B_{\delta_x} (x)\}.
$$ 
For every $j\in J$, we define $E_j(u)\in SC^\alpha ( B_{\delta_x}(x))$ to be the extension of $u$ to the ball $ B_{\delta_x} (x)$, containing $\supp \chi_j$, which is given by Theorem \ref{t1}, for $j\in J$. 
Set 
\begin{equation}\label{eq:eux}
E(u)(y):=\sum_{j\in J} \chi_j(y)  E_j(u)(y) +\sum_{j\in \N\setminus J} \chi_j(y) u(y), \quad y\in \R^n.
\end{equation}
We observe that, by Condition (D) above,  we deduce 
$$
E(u)(y)=\sum_{j\in\N} \chi_j(y)u(y)=u(y),\quad \forall y\in\overline{\Omega}.
$$
Furthermore, we have that $E(u)\in SC_{loc}^\alpha (\Omega')$.  Indeed, let $K\subset \Omega'$ be a compact set. Then, by Condition (B) above, only for finitely many $j$ $\supp \chi_j \cap K\not=\emptyset$ (i.e. the sum in \eqref{eq:eux} is finite).  Then, recalling that the product of a smooth nonnegative function with a semiconcave one is semiconcave and  that the sum of finitely many semiconcave functions is a semiconcave function, we deduce that $E(u)\in SC^\alpha (K)$.

In order to complete our proof, it remains to show that, for every $x\in \partial \Omega$, $\co D^*E(u)(x)=\co D^*u(x)$. 
We observe that, by construction, $D^*E(u)(x)\supset D^*u(x)$. Then, our proof reduces to show that 
\begin{equation}\label{eq:codecodu}
D^*E(u)(x)\subset \co D^*u(x), \qquad \forall x\in \partial \Omega .  
\end{equation}
Then, let $x\in \partial\Omega$ and let $x_h\in \Omega'\setminus \overline{\Omega}$, $h\in \mathbb N$, be a sequence of points of differentiability for $E(u)$ such that $x_h$ converges to $x$ and $DE(u)(x_h)$ converges to a suitable vector $p$. Hence, the proof reduces to show that $p\in\co D^*u(x)$. 
We observe that the sum, in the definition of $E(u)$, is a finite sum (on a compact set only finitely many $j$ are involved, say $j_1,\ldots , j_s$). 

Let us suppose that 
\begin{multline}\label{eq:fa}
\text{ each } E_{j_k}(u) \text{  is differentiable at  }x_h,\text{  for every }h \text{  and }
\\
k=1,\ldots ,s.
\end{multline}
Then, we find that 
\begin{multline}\label{eq:dei}
D(    \chi_{j_k} E_{j_k}(u)) (x_h)\\
=D \chi_{j_k} (x_h)  E_{j_k}(u)(x_h) + \chi_{j_k} (x_h) DE_{j_k}(u)(x_h).
\end{multline}
Hence, taking the sum w.r.t. $k$ in both the sides of the identity \eqref{eq:dei}, and recalling Condition (D) above, we deduce that 
$$
DE(u)(x_h)=\sum_{ k=1}^s\chi_{j_k} (x_h) DE_{j_k}(u)(x_h).
$$
Then, possibly taking a subsequence of $x_h$, we may assume that, as $h\to \infty$,  $DE_{j_k}(u)(x_h)$ converges to a suitable $p_k$, with $p_k\in D^*u(x)$  (by Theorem \ref{t1}), i.e. we conclude that 
$$
p=\sum_{ k=1}^s\chi_{j_k} (x) \, p_k \in \co D^*u(x). 
$$
It remains to discuss Assumption \eqref{eq:fa} which, as shown in the next result, is automatically satisfied. 
\begin{lemma}
Let $v_1,\ldots ,v_\ell \in SC^\alpha (B)$, for suitable $\ell \in \N$, $\alpha \in ]0,1]$ and $B\subset \R^n$. 
Then, if $\sum_{j=1}^\ell v_j$ is differentiable at $x_0$, then each $v_j$ is differentiable at $x_0$.  
\end{lemma} 
 \begin{proof} 
 It suffices to verify the statement in the case of $\ell =2$ (being the general case a direct consequence of an elementary argument by induction). 
 We have that 
 $$
v_1(x)= (v_1(x)+v_2(x))+(-v_2(x)) 
 $$
 and we observe that $-v_2$ is a semiconvex function. Then, we find that 
 $$
 D^+v_1(x_0)=D(v_1+v_2)(x_0)+D^+(-v_2)(x_0).
 $$
 (Here, we are using a calculus rule for the superdifferential which applies because of we are assuming that $v_1+v_2$ is differentiable at $x_0$.)  
 Hence, we deduce that $D^+(-v_2)(x_0)$ is nonempty if and only if $-v_2$ is differentiable at $x_0$ (see e.g. \cite{CS}). Then, we deduce that $v_1$ and $v_2$ are differentiable at $x_0$. 
 \end{proof} 
 This completes our proof. 
 \section{Examples}
 In order to clarify that the key information provided by Theorem \ref{sing} is a direction of ``propagation''  $\theta$,  let us give some examples. 
We begin by showing that, under Condition \eqref{eq:h}, a singularity may  propagate in the exterior of the set $\Omega$. 
\begin{example}\em
Let $\Omega \subset  \{  (x_1,x_2)\in \R^2\, |\, x_1>0\}$ be an open set, with nonemtpy boundary, such that      
$$
\Omega \supset \{ (x_1,x_2)\in \R^2\, |\, x_1>0\quad\text{ and }x_1^2+x_2^2<1\},
$$
and let 
$$
u(x)=-|x|, \qquad \forall x\in \bar{\Omega} .
$$

Let us compute $D^*u(0)$. For $x\in \Omega$, we have 
$$
Du(x)=-\frac x{|x|}. 
$$
For every unit vector $v=(v_1,v_2)$, with $v_1 >0$, let $x_h\in \Omega$ be such that 
$$
x_h\to 0\quad\text{ and }\quad \frac{x_h}{ |x_h|} \to v,\quad \text{ as }h\to\infty . 
$$
Then, we find that $Du(x_h)=- x_h/|x_h|\to -v $, as $h\to\infty$,  i.e. 
$$
D^*u(0)=\{ p\in \R^2\,  | \, |p|=1\quad \text{ and }\quad p_1 \leq 0    \}.
$$
(In the last identity we used the fact that $D^*u$ is a closed set.) 
Furthermore, we have that 
$$
\co D^*u(0)\setminus D^*u(0)=\{ p\in \R^2\,  | \, |p|<1\quad \text{ and }\quad p_1 \leq 0    \} \not=\emptyset
$$ 
and that, by Theorem \ref{sing} with $p_0=0$ and $\theta =( -1 ,0)$, we deduce that the singularity at $0$ propagates along the negative $x_1$ axis.
Let us verify the propagation result by a direct computation of the extension $E(u)$. We point out that, since $u$ is a concave function, we can take $C=0$ in 
the definition of $E(u)$ \eqref{eu}.      
We have that 
\begin{align*}
&E(u)(x)\\
&=\inf    \Big (   \{ \langle \eta ,x\rangle +|x|^2 \ | \   \text{ for }  |\eta |= 1,\,  \eta_1 \leq 0 \} 
\\
&\qquad \qquad\cup   \{   -|y|-\frac  1{|y|}\langle y,x-y\rangle +|x-y|^2\, |\, \text{ for }   0<|y|\le 1,\, y_1\geq 0 \} \Big).
\end{align*}
In order to evaluate the infimum above, we observe that 
\begin{equation}\label{min}
 -|y|-\frac  1{|y|}\langle y,x-y\rangle +|x-y|^2=-\frac  1{|y|}\langle y,x\rangle +|x-y|^2
 \end{equation}
and that, since $x_1<0$ and $y_1\geq 0$,  
$$
-\frac  1{|y|}\langle y,x\rangle \geq -|x_2|.  
$$
(Analogously we have that $\langle \eta ,x\rangle \geq -|x_2|$.) 
Then, we find that 
\begin{equation}\label{eq:eu}
E(u)(x)\geq -|x_2|+x_1^2\qquad \text{ for $x\in B_1(0)$ with }x_1< 0 
\end{equation}
Hence, by \eqref{eq:eu} and \eqref{min} with $y=(0,x_2)$,  we conclude that  
$$
E(u)(x)
=
\begin{cases}
-|x|, \qquad &\text{ for $x\in B_1(0)$ with }x_1\geq 0,
\\
-|x_2|+x_1^2\qquad &\text{ for $x\in B_1(0)$ with }x_1< 0.
\end{cases}
$$
Let us remark that, even if $u$ is a concave function in $\Omega$, we have that $E(u) \in SC^1(B_1(0))$ is not concave  (in particular 
$u(x)\not= -|x|$, for $x\in B_1(0)$ with $x_1<0$).
\end{example}
In the second example we show that Condition \eqref{eq:h} can be satisfied and a singularity may propagate in the interior of the set $\Omega$. 
\begin{example}\em
Let $\Omega \subset  \{  (x_1,x_2)\in \R^2\, |\, x_1>0\}$ be an open set, with nonempty boundary, such that      
$$
\Omega \supset \{ (x_1,x_2)\in \R^2\, |\, x_1>0\quad\text{ and }x_1^2+x_2^2<1\},
$$
and let  
$$
u(x)=-|x_2|, \qquad \forall x\in \bar{\Omega} .
$$

Then, we find   that 
$$
D^*u(0)=\{ (0,p_2)\in \R^2\,  | \, p_2=\pm 1 \}
$$
 and
$$
\co D^*u(0)\setminus D^*u(0)=\{0\}\times [-1,1]\setminus \{ (0,\pm 1)\}= \{0\}\times ]-1,1[      \not=\emptyset . 
$$ 
Even if in this example it is clear that and that all the points of the form $(x_1,0)$  are points of nondifferentiability for $u$, we observe that one may apply Theorem \ref{sing} with   $p_0=0$ and $\theta =(1,0)$ to deduce that the singularity at $0$ propagates in the interior of the set $\Omega$. 
 
Let us also observe that 
$$
E(u)(x)
=
\begin{cases}
-|x_2|, \qquad &\text{ for $x\in B_1(0)$ with }x_1\geq 0,
\\
-|x_2|+x_1^2,\qquad &\text{ for $x\in B_1(0)$ with }x_1< 0.
\end{cases}
$$
\end{example}
In the next example we show that Assumption \eqref{eq:h} is only a sufficient condition for the propagation of singularities.
\begin{example}\em
Let $\Omega \subset  \{  (x_1,x_2)\in \R^2\, |\, x_1>0\}$ be an open set,  with nonempty boundary, such that      
$$
\Omega \supset \{ (x_1,x_2)\in \R^2\, |\, x_1>0\quad\text{ and }x_1^2+x_2^2<1\},
$$
and let  
$$
u(x)=-\sqrt{x_1^4+x_2^2}, \qquad \forall x\in \bar{\Omega} .
$$

Then, we find   that 
$$
D^*u(0)=\{0\}\times [-1,1]
$$
 and
$$
\co D^*u(0)\setminus D^*u(0)=\emptyset . 
$$ 
Let $x\in B_1(0)$ with $ x_1<0$, we have  
\begin{align*} 
E(u)(x)
&=\inf_{  y\in B_1(0),\, y_1\geq 0,\, p\in D^*u(y)} [ -\sqrt{y_1^4+y_2^2}+ \langle p, x-y\rangle +|x-y|^2]
\\
&\le \inf_{p_2,y_2\in [-1,1]}    [ -|y_2|+ p_2 (x_2-y_2)+x_1^2+(x_2-y_2)^2]    \le -|x_2|+x_1^2 .
\end{align*}
 Furthermore, if either $y_1>0$  (i.e. $x_1y_1<0$) or $y_1=0$ and $y_2\not= 0$, we have 
 \begin{align*}
 &-\sqrt{y_1^4+y_2^2}-\frac 1 {\sqrt{y_1^4+y_2^2}}    \langle (2y_1^3,y_2)     , x-y\rangle +|x-y|^2
 \\
 &=\frac {y_1^4-x_2y_2-2x_1y_1^3} {\sqrt{y_1^4+y_2^2}} +|x-y|^2\\
  &\geq \frac {-x_2y_2} {\sqrt{y_1^4+y_2^2}} +|x-y|^2
 \\
& \geq -|x_2|+|x-y|^2\geq  -|x_2|+x_1^2. 
 \end{align*}
 Finally, if $y=0$ and $p\in D^*u(0)$, we find 
$$
  -\sqrt{y_1^4+y_2^2}+ \langle p, x-y\rangle +|x-y|^2\geq -|x_2|+|x|^2.
 $$
 Then, we conclude that, for $x\in  B_1(0)$, 
 $$
E(u)(x)
=
\begin{cases}
-\sqrt{x_1^4+x_2^2}, \quad &\text{ if }x_1\geq 0 
\\
-|x_2|+x_1^2,\quad &\text{ if }x_1< 0 .
\end{cases}
$$
In other words, condition \eqref{eq:h} is not satisfied but the singularity of $u$ at the origin propagates along the negative $x_1$ axis.
\end{example} 
 \section{Proof of Theorems \ref{app},  \ref{app*} and \ref{sing}.}

 \subsection{Proof of Theorems \ref{app} and \ref{app*} }

 We provide the proof only of the local approximation result Theorem \ref{app}. We point out that the proof of Theorem \ref{app*} follows exactly the same lines using Theorem \ref{t1*} instead of the local extension result.

  Fix $x_0\in \partial \Omega$. Let $\delta >0$ and $E(u)$ the number and the extension given by Theorem \ref{t1} respectively. 
Let $\chi\in C^\infty_0(B_1(0))$ be a nonnegative function such that $\int \chi \, dx=1$ and define  
\begin{equation}\label{eq:uh}
u_h(x)=\int E(u)(x+y/h)\chi (y)\, dy,
\end{equation}
where $h>\frac 2\delta $ is a positive integer. We observe that 
$$
| x+y/h-x_0|\le \delta 
$$
for every $x\in B_{\delta /2}(x_0)$ and $y\in B_1(0)$.  Then $u_h$ is well-defined for $x\in B_{\delta /2}(x_0)$ and $u_h \rightarrow u$ uniformly on $\bar{B}_{\delta /2}(x_0)\cap \bar{\Omega}$. In order to complete our proof it remains to show that $u_h$ satisfies a uniform semiconcavity estimate. Let $\lambda \in [0,1]$ and $x_1,x_2 \in B_{\delta/2}(x_0)$, then we have 
\begin{align*}
&\lambda u_h(x_1)+(1-\lambda )u_h(x_2) -u_h(\lambda x_1 +(1-\lambda )x_2 )
\\
&=\int  [ \lambda E(u)(x_1+y/h)+   (1-\lambda )      E(u)(x_2+y/h)  
\\
 &\qquad  - E(u) (   \lambda x_1 +(1-\lambda )x_2 +y/h) ]\,  \chi (y)\, dy\leq C \lambda (1-\lambda ) |x_1-x_2|^{1+\alpha }.
\end{align*}
This completes our proof. 

\subsection{Proof of Theorem \ref{sing}}
 
The proof is based on an abstract propagation result given in \cite{A1} (see also \cite{AC} for an earlier form of the result). 
 
Let $x_0\in \partial \Omega  \cap \Sigma (u)$, in particular the set $D^*u(x_0)$ has at least two elements. 
Let us consider the extension given by Theorem \ref{t1}. We have that 
\begin{equation}\label{eq:sc}
D^+E(u)(x_0)=\co D^*E(u)(x_0)\quad (=\co D^*u(x_0)). 
\end{equation}
By Assumption \eqref{eq:h}, there exist $p_0\in \partial D^+E(u)(x_0)\setminus  D^*E(u)(x_0)$ and a vector $-\theta$ in  the normal cone to $D^+E(u)(x_0)$ at $p_0$.

Then, by Theorem 4.2. of \cite{A1}, we deduce that there exist a positive number $\sigma$ and a map $[0,\sigma ]\ni s\to x(s)$ such that 
\begin{enumerate}
\item $x(0)=x_0$ and $\lim_{s\to 0} x(s)=x_0$;  
\item   $x(s)\not=x_0$, for every $s\in [0,\sigma ]$;
\item $x(s)\in \Sigma (E(u))$, for every $s\in [0,\sigma ]$;
\item $x(s)=x_0+s\theta +o(s)$ with $o(s)/s\to 0$ as $s\to 0^+$. 
\end{enumerate}
 This completes our proof. 

\section{Declarations}
\begin{itemize}
\item {\bf Conflict of interest:} On behalf of all authors, the corresponding author states that there is no conflict of interest. 
\item{\bf Funding:} This work was partly supported by the National Group for Mathematical Analysis, Probability and Applications (GNAMPA) of the Italian Istituto Nazionale di Alta Matematica ``Francesco Severi''; moreover, the third author acknowledges support by the Excellence Department Project awarded to the Department of Mathematics, University of Rome Tor Vergata, CUP E83C18000100006. 
\item{\bf Acknowledgments:} We warmly thank the reviewer for carefully reading the manuscript. 
\end{itemize}

\end{document}